\documentclass[12pt]{amsart}
\usepackage{amssymb}
\usepackage[mathscr]{eucal}
\usepackage{epsf}
\usepackage{epsfig}
\usepackage{color}
\DeclareGraphicsExtensions{.pstex,.eps,.epsi,.ps}

 \newtheorem{thm}{Theorem}[section]
 \newtheorem{cor}[thm]{Corollary}
 
 \newtheorem{prop}[thm]{Proposition}
 \theoremstyle{definition}
 
 \theoremstyle{remark}
 
 \numberwithin{equation}{section}

 \newcommand{\Real}{\mathbb{R}}

 \newcommand{\hor}{\text{hor}}
 
 \newcommand{\tr}{\textbf{tr}}
 \newcommand{\di}{\textbf{div}}
 \newcommand{\vol}{\textbf{vol}}
 \newcommand{\ric}{\textbf{Rc}}

 \newcommand{\Rm}{\textbf{Rm}}

\newcommand{\e}{\epsilon}

\newcommand{\LD}{\mathcal L}
\newcommand{\cA}{\mathcal A}
\newcommand{\cF}{\mathcal F}

\newcommand{\cJ}{\mathcal J}

\begin{document}

\title[On MCP without Ricci curvature lower bound]{On measure contraction property without Ricci curvature lower bound}

\author{Paul W.Y. Lee}
\email{wylee@math.cuhk.edu.hk}
\address{Room 216, Lady Shaw Building, The Chinese University of Hong Kong, Shatin, Hong Kong}

\date{\today}

\maketitle

\begin{abstract}
Measure contraction properties $MCP(K,N)$ are synthetic Ricci curvature lower bounds for metric measure spaces which do not necessarily have smooth structures. It is known that if a Riemannian manifold has dimension $N$, then $MCP(K,N)$ is equivalent to Ricci curvature bounded below by $K$. On the other hand, it was observed in \cite{Ri} that there is a family of left invariant metrics on the three dimensional Heisenberg group for which the Ricci curvature is not bounded below. Though this family of metric spaces equipped with the Harr measure satisfy $MCP(0,5)$.

In this paper, we give sufficient conditions for a $2n+1$ dimensional weakly Sasakian manifold to satisfy $MCP(0,2n+3)$. This extends the above mentioned result on the Heisenberg group in \cite{Ri}.
\end{abstract}


\section{Introduction}

In the past decade, there is a surge of interest in studying synthetic Ricci curvature lower bounds. These are reformulations of Ricci curvature lower bounds on Riemannian manifolds without using the underlying smooth structure. As a consequence, they can be used as the definitions of Ricci curvature lower bounds on more general metric measure spaces.

There are quite a few synthetic Ricci curvature lower bounds defined via different approaches. This includes the one in \cite{BaEm} via the formalism of Dirichlet forms, the one in \cite{LoVi1,St1,St2} via the theory of optimal transportation, and the one in \cite{Ol} via coupling of Markov chains.

In this paper, we consider another synthetic Ricci curvature lower bound, called measure contraction property $MCP(K,N)$, discussed in \cite{St2,Oh1}.  Here, we recall that a length space $(M,d)$ equipped with a measure $\mu$ satisfies $MCP(0,N)$ if, for each Borel set $U_0$ and each point $x_0$ in $M$, the contraction $U_t$ of $U_0$ along geodesics ending at $x_0$ satisfies
\[
\mu(U_t)\geq (1-t)^N\mu(U_0).
\]
The condition $MCP(K,N)$ is defined in a similar way. For Riemannian manifolds with dimension $N$, the condition $MCP(K,N)$ is equivalent to the Ricci curvature bounded below by $K$. On the other hand, it was observed in \cite{Ri} that there is a family of left invariant metrics on the three dimensional Heisenberg group for which the Ricci curvature is not bounded below. Though this family of metric spaces equipped with the Harr measure satisfy $MCP(0,5)$.

In this paper, we give sufficient conditions on a family of Riemannian manifolds, called weakly Sasakian manifolds, of dimension $2n+1$ which guarantee that the condition $MCP(0,2n+3)$ holds. More precisely, let $M$ be a contact manifold of dimension $2n+1$ equipped a contact form $\eta$ and a Reeb field $V$. Let $\cJ$ be a $(1,1)$-tensor which is almost complex on the distribution $\ker\eta$ and $\cJ V=0$. The Riemannian metric $\left<\cdot,\cdot\right>$ is defined by the 2-form $d\eta$ and the tensor $\cJ$ on $\ker\eta$. Outside $\ker\eta$, the Riemannian metric is defined by $|V|=1$. On such a manifold, one can define a convenient connection, called the Tanaka-Webster connection. The corresponding curvature tensor, denoted by $\overline\Rm$, is called the Tanaka-Webster curvature (see Section \ref{wkSa} for the detail).

The geometric structure $(M,\cJ,V,\eta,\left<\cdot,\cdot\right>)$ is a Sasakian manifold if additional compatibility and integrability conditions are satisfied (see Section \ref{wkSa} for the precise definition). We call $(M,\cJ,V,\eta,\left<\cdot,\cdot\right>)$ a weakly Sasakian manifold if all the above mentioned conditions except $|V|=1$ are satisfied. We show that the Ricci curvature $\ric$ blows up in some directions as $|V|=\e\to\infty$. On the other hand, we show that

\begin{thm}\label{main}
Let $(M,\cJ,V,\eta,\left<\cdot,\cdot\right>)$ be a weakly Sasakian manifold of dimension $2n+1$ such that $|V|$ is constant. Assume that the Tanaka-Webster curvature $\overline\Rm$ satisfies
\begin{enumerate}
\item $\left<\overline\Rm(Jv,v)v,Jv\right>\geq 0$,
\item $\sum_{i=1}^{2n-2}\left<\overline\Rm(w_i,v)v,w_i\right>\geq 0$,
\end{enumerate}
for any orthonormal basis $\{v,Jv,w_1,...,w_{2n-2}\}$ of $\ker\eta$. Then the metric measure space $(M,d,\vol)$ satisfies $MCP(0,2n+3)$, where $d$ and $\vol$ are, respectively, the Riemannian distance and the Riemannian volume of $\left<\cdot,\cdot\right>$.
\end{thm}

Note that the curvature conditions in Theorem \ref{main} are satisfied by the Heisenberg group. In fact, all inequalities become equalities in this case.

Note also that, under the same assumptions as in Theorem \ref{main}, it was shown in \cite{LeLiZe} that $(M,d_{CC},\vol_P)$ satisfies $MCP(0,2n+3)$, where $d_{CC}$ is the Carnot-Caratheordory distance and $\vol_P$ is the Popp measure (see also \cite{Ju,AgLe} for the earlier results).

Metric measure spaces satisfying measure contraction property\\ $MCP(0,N)$, in particular the ones defined in Theorem \ref{main}, satisfy doubling property and Poincar\'e inequality.

\begin{cor}(Doubling)\label{doubling}
Assume that the conditions in Theorem \ref{main} hold. Then there is a constant $C>0$ such that
\[
\vol(B_{x}(2R))\leq C\vol(B_{x}(R))
\]
for all $x$ in $M$ and all $R>0$, where $B_x(R)$ is the ball of radius $R$ centered at $x$.
\end{cor}

\begin{cor}(Poincar\'e inequality)\label{Poincare}
Assume that the conditions in Theorem \ref{main} hold.  Then, for each $p>1$, there is a constant $C>0$ such that
\[
\begin{split}
&\int_{B_{x}(R)}\left|f(x)-\frac{1}{\vol({B_{x}(R)})}\int_{B_{x}(R)}f(x)d\vol(x)\right|^pd\vol(x) \\
&\leq CR^p\int_{B_{x}(R)}|\nabla f|^pd\vol(x).
\end{split}
\]
\end{cor}
Here Corollary \ref{doubling} easily follows from the measure contraction property. For a proof of Corollary \ref{Poincare} which relies on a result in \cite{Je}, see \cite{LeLiZe}.

With the doubling property and the Poincar\'e inequality, numerous results follow. For instance, it follows from the above corollaries and the results in \cite{Mo0, CoHoSa} that

\begin{cor}(Harnack inequality)
Assume that the conditions in Theorem \ref{main} hold.  Then, for each $p>1$, there is a constant $C>0$ such that any positive solution to the $p$-Laplace equation $\di_\vol(|\nabla f|^{p-2}\nabla f)=0$ on $B_x(R)$ satisfies
\[
\sup_{B_{x}(R/2)}f\leq C\inf_{B_{x}(R/2)}f.
\]
\end{cor}

\begin{cor}(Liouville theorem)
Assume that the conditions in Theorem \ref{main} hold. Then any non-negative solution to the $p$-Laplace equation is a constant.
\end{cor}

The following parabolic Harnack inequality also holds (see \cite{Mo1,Mo2,Gr,Sa1}). 

\begin{cor}(Parabolic Harnack inequality)
Assume that the conditions in Theorem \ref{main} hold.  Then, for each $R>0$, there is a constant $C>0$ such that any positive solution to the heat equation $\dot f=\Delta f$ on $(s-r^2,s)\times B_x(R)$ with $0<r<R$ satisfies
\[
\sup_{(s-\frac{3}{4}r^2,s-\frac{1}{2}r^2)\times B_{x}(R/2)}f\leq C\inf_{(s-\frac{1}{4}r^2,s)\times B_{x}(R/2)}f.
\]
for all points $x$ in $M$.
\end{cor}

For a converse result of the above corollary, see also \cite{KuSt,Gr,Sa1}. The above parabolic Harnack inequality is also equivalent to a two sided Gaussian bound for the heat kernel (see \cite{FoSt}). 

\begin{cor}(Two-sided Gaussian bound)
Assume that the conditions in Theorem \ref{main} hold.  Then there are positive constants\\ $C_1, C_2, C_3, C_4$ such that the heat kernel $h$ satisfies
\[
\frac{C_1}{\vol(B_{\sqrt t}(x))}e^{-\frac{C_2d(x,y)^2}{t}}\leq h_t(x,y)\leq \frac{C_3}{\vol(B_{\sqrt t}(x))}e^{-\frac{C_4d(x,y)^2}{t}}
\]
for all points $x$ and $y$ in $M$.
\end{cor}

Finally, we remark that there are also consequences following from Corollary \ref{doubling} and \ref{Poincare} about quasi-regular mappings. For this, see \cite{CoHoSa} and references therein.

The paper is organized as follows. In Section \ref{wkSa}, we introduce the weakly Sasakian manifolds and summarize some facts that are needed for this paper. In Section \ref{conjMC}, we begin the proof of Theorem \ref{main} by introducing one of the key ingredients of the proof, a moving frame adapted to the given geometry defined along a geodesic. We also rewrite the measure contraction property as estimates on solutions of a matrix Riccati equation using this moving frame. This approach was also used by the author in various other situations (see \cite{Le1,Le2}). In Section \ref{Hei}, the case of the Heisenberg group is discussed. The proof of Theorem \ref{main} is summarized in Section \ref{proofmain}. Finally, the proofs of the results mentioned in Section \ref{wkSa} are discussed in the appendix.

\smallskip

\section*{Acknowledgements}

The whole work started from a discussion with Professor Ludovic Rifford on the case of the Heisenberg group in \cite{Ri}. The author would like to thank him for suggesting the problem and his encouragement.

This material is partly based upon work supported by the National Science Foundation under Grant No. 0932078 000, while the author was in residence at
the Mathematical Science Research Institute in Berkeley, California, during Fall 2013. The author was also supported by the Research Grant Council of Hong Kong (RGC Ref. No. CUHK404512).

\smallskip

\section{Weakly Sasakian manifolds}\label{wkSa}

In this section, we introduce what we call weakly Sasakian manifolds and discuss some of the properties that are needed in this paper.

Let $\eta$ be a contact form on a manifold $M$ of dimension $2n+1$. This means that the restriction of the two-form $d\eta$ to the distribution $\ker\eta$ is symplectic. Let $V$ be the Reeb field defined by $\eta(V)=1$ and $d\eta(V,\cdot)=0$. Let $J$ be a $(1,1)$-tensor satisfying $\cJ V=0$ and $\cJ^2X=-X$ for all vector field $X$ contained in the distribution $\ker\eta$. Let $\left<\cdot,\cdot\right>$ be a Riemannian metric such that
\begin{equation}\label{conmet}
d\eta(X_1,X_2)=\left<X_1,\cJ X_2\right>.
\end{equation}
Note that this implies, in particular, that the Reeb field $V$ is orthogonal to the distribution $\ker\eta$. We call the structure $(\cJ,V,\eta,\left<\cdot,\cdot\right>)$ weakly contact metric structure. We also say the structure $(\cJ,V,\eta,\left<\cdot,\cdot\right>)$ is weakly Sasakian if the following holds for all vector fields $Y_1$ and $Y_2$:
\[
\begin{split}
&[Y_1,Y_2]+\cJ[\cJ Y_1,Y_2]+\cJ[Y_1,JY_2]-[\cJ Y_1,\cJ Y_2]\\
&=(Y_1\cdot\eta(Y_2)-Y_2\cdot\eta(Y_1))V.
\end{split}
\]

In other words,
\begin{equation}\label{Nij}
\begin{split}
&d\eta(Y_1,Y_2)V\\
&=-\cJ^2[Y_1,Y_2]+\cJ[\cJ Y_1,Y_2]+\cJ[Y_1,\cJ Y_2]-[\cJ Y_1,\cJ Y_2].
\end{split}
\end{equation}

Note that the structure $(\cJ,V,\eta,\left<\cdot,\cdot\right>)$ is Sasakian if $|V|=1$. In this paper, we consider weakly Sasakian manifolds such that the length $|V|$ of the Reeb field $V$ is constant. The proof of the following result is contained in the appendix.

\begin{prop}\label{ind}
Assume that the structure $(\cJ,V,\eta,\left<\cdot,\cdot\right>)$ is weakly Sasakian and $|V|=\e$. Then
\begin{enumerate}
\item $\eta(Y)=\frac{1}{\epsilon^2}\left<V,Y\right>$
\item $\LD_V\cJ=0$,
\item $\LD_Vg=0$,
\item $\nabla_YV=-\frac{\epsilon^2}{2}\cJ Y$,
\item $\nabla_{X_1}\cJ(X_2)=\frac{\left<X_1,X_2\right>}{2}V$,
\item $\nabla_X\cJ(V)=-\frac{\epsilon^2}{2}X$,
\item $\nabla_V\cJ=0$,
\item $(\nabla_{X_1}X_2)_\hor$ is independent of $\e$,
\item $\nabla_{X_1}X_2=(\nabla_{X_1}X_2)_\hor+\frac{1}{2}\left<\cJ X_1,X_2\right>V$,
\item $\nabla_YV=-\frac{\epsilon^2}{2}\cJ Y$,
\item $\nabla_V X_1=([V,X_1])_\hor-\frac{\e^2}{2}\cJ X_1$,
\end{enumerate}
for all vector fields $X_1$, $X_2$, and $Y$ such that $X_1$ and $X_2$ are contained in the distribution $\ker\eta$. Here $Y_\hor$ denotes the orthogonal projection of the vector field $Y$ onto the distribution $\ker\eta$.
\end{prop}

Let $\bar\nabla$ be the connection defined by
\[
\bar\nabla_{Y_1}Y_2=\nabla_{Y_1}Y_2+\frac{\left<V,Y_2\right>}{2}\cJ Y_1-\frac{1}{2}\left<\cJ Y_1,Y_2\right>V+\frac{\left<V,Y_1\right>}{2}\cJ Y_2.
\]
Note that $\bar\nabla$ is the Tanaka-Webster connection when $\e=1$.

\begin{prop}
The connection $\bar\nabla$ is independent of $\e$.
\end{prop}

\begin{proof}
Note that the following formula holds for all vector fields $Y_1$ and $Y_2$
\[
\begin{split}
&\bar\nabla_{Y_1}Y_2=\nabla_{Y_1}Y_2+\frac{\left<V,Y_2\right>}{2}\cJ Y_1-\frac{1}{2}\left<\cJ Y_1,Y_2\right>V+\frac{\left<V,Y_1\right>}{2}\cJ Y_2\\
&=\eta(Y_1)[V,Y_2]+\nabla_{(Y_1)_\hor}Y_2+\frac{\left<V,Y_2\right>}{2}\cJ Y_1 -\frac{1}{2}\left<\cJ Y_1,Y_2\right>V\\
&=\eta(Y_1)[V,Y_2]+(\nabla_{(Y_1)_\hor}(Y_2)_\hor)_\hor+\LD_{(Y_1)_\hor}(\eta(Y_2))V.
\end{split}
\]
Therefore, the result follows from Proposition \ref{ind}.
\end{proof}

Let $\Rm$ and $\overline\Rm$ be the curvature tensors defined by the connections $\nabla$ and $\bar\nabla$, respectively. The two curvatures are related as follows (see Appendix for the proof).

\begin{prop}\label{Rm}
Assume that the structure $(J,V,\eta,\left<\cdot,\cdot\right>)$ is weakly Sasakian and $|V|=\e$. Then
\begin{enumerate}
\item $\Rm(Y_1,Y_2)V=\frac{\e^2\left<Y_2,V\right>}{4}(Y_1)_\hor-\frac{\e^2\left<Y_1,V\right>}{4}(Y_2)_\hor$,
\item $\overline{\Rm}(X_2,X_3)X_1=\Rm(X_2,X_3)X_1+\frac{\e^2\left<\cJ X_3,X_1\right>}{4}\cJ X_2\\ -\frac{\e^2\left<\cJ X_2,X_1\right>}{4}\cJ X_3-\frac{\e^2\left<\cJ X_2,X_3\right>}{2}\cJ X_1$,
\item $\overline\Rm(Y_1,Y_2)V=0$,
\item $\overline\Rm(X_1,V)X_2=\Rm(X_1,V)X_2+\frac{\e^2}{4}\left<X_1,X_2\right>V$,
\item $(\overline\Rm(X_1,V)X_2)_\hor=0$,
\end{enumerate}
for all vector fields $X_1$, $X_2$, $Y_1$, and $Y_2$ such that $X_1$ and $X_2$ are contained in the distribution $\ker\eta$.
\end{prop}

Let $v_0,v_1,...,v_{2n}$ be an orthonormal frame such that $v_0=\frac{1}{\e}V$, $v_1=\frac{1}{|Y_\hor|}Y_\hor$, $v_2=\frac{1}{|Y_\hor|}\cJ Y_\hor$, and $Jv_{2k-1}=v_{2k}$ for each $k=1,...,n$. The following is a consequence of Proposition \ref{Rm}.

\begin{prop}\label{conseq}
Assume that the structure $(J,V,\eta,\left<\cdot,\cdot\right>)$ is weakly Sasakian and $|V|=\e$. Then, for each $i, j \neq 0$,
\begin{enumerate}
\item $\left<\Rm(v_i,Y)Y,v_0\right>=-\frac{\e\left<Y,V\right>|Y_\hor|}{4}\delta_{i1}$,
\item $\left<\Rm(v_0,Y)Y,v_0\right>=\frac{\e^2}{4}|Y_\hor|^2$
\item $\left<\Rm(v_i,Y)Y,v_j\right>,\\
=\frac{\left<Y,V\right>^2}{4}\delta_{ij}-\frac{3\e^2\delta_{i2}\delta_{j2}|Y_\hor|^2}{4}+\left<\overline{\Rm}(v_i,Y)Y,v_j\right>$\\
\item $\ric(Y,Y)=\frac{n\left<Y,V\right>^2}{2}-\frac{3\e^2|Y_\hor|^2}{4}+\overline{\ric}(Y,Y)$,
\end{enumerate}
where $\ric(Y,Y)$ and $\overline\ric(Y,Y)$ are the traces of $v\mapsto\left<\Rm(v,Y)Y,v\right>$ and $v\mapsto\left<\overline\Rm(v,Y)Y,v\right>$, respectively.
\end{prop}

Note that, for each tangent vector $Y$ with $Y_\hor\neq 0$, $\ric(Y,Y)\to -\infty$ as $\e\to\infty$.

\smallskip

\section{On conjugate points and measure contraction}\label{conjMC}

From now on, we assume that the structure $(\cJ,V,\eta,\left<\cdot,\cdot\right>)$ on the manifold $M$ is weakly Sasakian with $|V|=\e$. In this section, we prove some preliminary results on conjugate points and measure contraction properties of these manifolds.

Let $t\mapsto\gamma_\e(t)$ be a family of geodesics parameterized by the variable $\e$ such that $\gamma_\e(0)=x$. It follows that $\frac{D^2}{dt^2}\gamma_\e(t)=0$ and so
\[
\begin{split}
0&=\frac{D}{d\e}\frac{D^2}{dt^2}\gamma_\e(t)\Big|_{\e=0}\\
&=\frac{D}{dt}\frac{D}{dt}\gamma_0'(t)+\Rm(\gamma'_0(t),\dot\gamma_0(t))\dot\gamma_0(t)\\
\end{split}
\]

Let $v_0(t)=\frac{1}{\epsilon}V(\gamma_0(t))$ and let $v_{2i}(t)=\mathcal Jv_{2i-1}(t)$. We also assume that $v_1(t)=\frac{1}{|(\dot\gamma_0(t))_H|}(\dot\gamma_0(t))_H$ and $v_2(t)=\frac{1}{|(\dot\gamma_0(t))_H|}\mathcal J\dot\gamma_0(t)$. It follows that
\[
\begin{split}
\dot v_0(t)&=\frac{1}{\epsilon}\nabla_{\dot\gamma_0(t)}V(\gamma_0(t))=-\frac{\epsilon|(\dot\gamma_0(0))_H|}{2}v_2(t),
\end{split}
\]
\[
\begin{split}
v_1(t)&=\frac{1}{|(\dot\gamma_0(t))_H|}\left(\dot\gamma(t)-\left<\dot\gamma(t),v_0(t)\right>v_0(t)\right),\\
&=\frac{1}{|(\dot\gamma_0(0))_H|}\left(\dot\gamma(t)-\left<\dot\gamma(0),v_0(0)\right>v_0(t)\right)
\end{split}
\]
\[
\begin{split}
\dot v_1(t)&=-\frac{\left<\dot\gamma(0),v_0(0)\right>}{|(\dot\gamma_0(0))_H|}\dot v_0(t)=\frac{\e\left<\dot\gamma(0),v_0(0)\right>}{2}v_2(t),
\end{split}
\]
and
\[
\begin{split}
\dot v_2(t)&=\nabla_{\dot\gamma_t}\mathcal J( v_1(t))+\mathcal J\dot v_1(t)\\
&=\frac{\e|(\dot\gamma(0))_H|}{2}v_0(t)-\frac{\epsilon\left<\dot\gamma(0),v_0(0)\right>}{2}v_1(t).
\end{split}
\]
Finally, we can choose $v_3(t),...,v_{2n}(t)$ such that $\dot v_i(t)$ is contained in the span of $v_0(t), v_1(t), v_2(t)$ for each $i=3,...,2n$. It follows that $\left<\dot v_i(t),v_j(t)\right>=-\left<v_i(t),\dot v_j(t)\right>=0$ for each $j=0,1,2$ and $i=3,...,2n$. Therefore, $\dot v_i(t)=0$.

Let $W(t)$ be the matrix defined by $\dot v(t)=W(t)v(t)$ and let $\nabla_Hf=(\nabla f)_\hor$. Since $|\dot\gamma(t)|$ and $\left<V(\gamma(t)),\dot\gamma(t)\right>$ are independent of $t$,
\[
W=\left(
       \begin{array}{cccc}
         0 & 0 & -\frac{\epsilon|\nabla_Hf_0|}{2} & 0 \\
         0 & 0 & \frac{\left<\nabla f_0, V\right>}{2} & 0 \\
         \frac{\epsilon|\nabla_Hf_0|}{2} & -\frac{\left<\nabla f_0, V\right>}{2} & 0 & 0 \\
         0 & 0 & 0 & O_{2n-2} \\
       \end{array}
     \right).
\]

Let $a(t)$ be the matrix defined by $\gamma'_0(t)=a(t)v(t)$. It follows that
\[
\frac{D}{dt}\gamma'_0(t)=\dot a(t)v(t)+a(t)Wv(t)
\]
and
\[
\begin{split}
-a(t)R(t)v(t)&=-\Rm(\gamma'_0(t),\dot\gamma_0(t))\dot\gamma_0(t)=\frac{D^2}{dt^2}\gamma'_0(t)\\
&=\ddot a(t)v(t)+2\dot a(t)Wv(t)+a(t)W^2v(t),
\end{split}
\]
where $R_{ij}(t)=\left<\Rm(v_i(t),\dot\gamma_0(t))\dot\gamma_0(t),v_j(t)\right>$.

Let $\cA(t)$ be solution of the following equation
\[
\begin{split}
\ddot \cA(t)+2\dot \cA(t)W+\cA(t)W^2+\cA(t)R(t)=0
\end{split}
\]
with initial conditions $\cA(0)=0$ and $\dot \cA(0)=I$.

Let $\cF(t)=\cA(t)^{-1}\dot \cA(t)+W$. Then
\begin{equation}\label{riccaticF}
\begin{split}
\dot \cF(t)&=-\cA(t)^{-1}\dot \cA(t)\cA(t)^{-1}\dot \cA(t)+\cA(t)^{-1}\ddot \cA(t)\\
&=-\cA(t)^{-1}\dot \cA(t)\cA(t)^{-1}\dot \cA(t)+\cA(t)^{-1}\ddot \cA(t)\\
&=-(\cF(t)-W)^2-2(\cF(t)-W)W-W^2-R(t)\\
&=-\cF(t)^2-\cF(t)W-W^T\cF(t)-R(t).
\end{split}
\end{equation}

Since $\gamma_0(0)$ and $\gamma_0(\tau)$ are conjugate along $\gamma$ if and only if $\left<\cF(t)\,v,v\right>\to -\infty$ as $t\to \tau$ for some vector $v$, we have the following

\begin{prop}\label{conjS}
Assume that $\gamma_0$ is a minimizing geodesic between its endpoints $\gamma_0(0)$ and $\gamma_0(1)$. Then $\tr\, \cF(t)$ stays bounded for all $t$ in $(0,1]$.
\end{prop}

Next, we consider the contraction of the measure $\vol$ along geodesics ending at the same point $x_0$. Let $d(x,x_0)$ be the Riemannian distance between the points $x$ and $x_0$. It is locally semi-concave and so twice differentiable Lebesgue almost everywhere. Let $\exp$ be the Riemannian exponential map and let $\varphi_t=\exp(t\nabla f_0)$, where $f_0(x)=-\frac{d^2(x,x_0)}{2}$. For the rest of this section, we discuss the volume contraction $\vol(\varphi_t(U))$, where $U$ is a fixed Borel set.

For this, let $x$ be a point where $f_0$ is twice differentiable. Let $v_0(t),...,v_{2n}(t)$ be a orthonormal frame along $\varphi_t(x)$ defined in the same way as the beginning of this section. Let $v(t)=(v_0(t),...,v_{2n}(t))^T$ and let $A(t)$ be the matrix defined by
\[
d\varphi_t(v(0))=A(t) v(t).
\]
It follows that
\[
\begin{split}
\frac{D}{dt}d\varphi_t(v(0))&=\dot A(t) v(t)+A(t) \dot v(t)\\
&=\left(\dot A(t) +A(t)W\right)v(t)\\
\end{split}
\]
and
\[
\begin{split}
&\frac{D^2}{dt^2}d\varphi_t(v(0))
=\left(\ddot A(t) +2\dot A(t) W+A(t)W^2\right)v(t).
\end{split}
\]

Let $f_t(y)=-\frac{d^2(y,x_0)}{2(1-t)}$. Then $f_t$ satisfies
 \[
 \dot f_t+\frac{1}{2}|\nabla f_t|^2=0
 \]
 at $x$ and so $\dot \varphi_t(x)=\nabla f_t(\varphi_t(x))$. Therefore, we also have
\[
\begin{split}
\frac{D}{dt}d\varphi_t(v_i(0))&=\frac{d}{ds}\nabla f_t(\varphi_t(\gamma_i(s)))\Big|_{s=0}\\
&=\nabla^2 f_t(d\varphi_t(v_i(0)))=\sum_{j,k}A_{ij}(t)F_{jk}(t)v_k(t)
\end{split}
\]
and
\[
\begin{split}
&\frac{D^2}{dt^2}d\varphi_t(v_i(0))=\frac{D}{dt}\frac{D}{ds}\nabla f_t(\varphi_t(\gamma_i(s)))\Big|_{s=0}\\
&=\Rm(\nabla f_t(\varphi_t),d\varphi_t(v_i(0)))\nabla f_t(\varphi_t))=-\sum_{j,k}A_{ij}(t)R_{jk}(t)v_k(t),
\end{split}
\]
where  $F_{ij}(t)=\left<\nabla^2f_t(v_i(t)),v_j(t)\right>$ and
\[
R_{ij}(t)=\left<\Rm(v_i(t),\nabla f_t(\varphi_t))\nabla f_t(\varphi_t),v_j(t)\right>.
\]

It also follows that
\[
\begin{split}
&-R(t)=A(t)^{-1}\ddot A(t) +2A(t)^{-1}\dot A(t) W+W^2.
\end{split}
\]
Hence,
\[
F(t)=A(t)^{-1}\dot A(t)+W
\]
and
\begin{equation}\label{riccatiF}
\begin{split}
\dot F(t)&=-R(t)-F(t)^2-F(t)W-W^TF(t).
\end{split}
\end{equation}

It also follows that $\det A(t)=e^{\int_0^t\tr F(s)ds}$. Hence, by applying \cite[Theorem 11.3]{Vi}, we obtain

\begin{prop}\label{measure}
\[
\int_{\varphi_t(U)}d\vol=\int_Ue^{\int_0^t\tr F(s)ds}d\vol.
\]
\end{prop}

Finally, we record the following formula.
\begin{prop}
Let $\bar R(t)$ be the matrix defined by
\[
\bar R_{ij}(t)=\left<\overline\Rm(v_i(t),\nabla f_t(\varphi_t))\nabla f_t(\varphi_t),v_j(t)\right>.
\]
Then $R(t)$ satisfies
\[
R(t)=\bar R(t)+\left(
       \begin{array}{cccc}
         b^2 & bc & 0 & 0 \\
         bc & c^2 & 0 & 0 \\
         0 & 0 & c^2-3b^2 & 0 \\
         0 & 0 & 0 & c^2I \\
       \end{array}
     \right)
\]
where $b=-\frac{\e|\nabla_H f_0|}{2}$ and $c=\frac{\left<\nabla f_0,V\right>}{2}$.
\end{prop}

\smallskip

\section{The Heisenberg group}\label{Hei}

In this section, we discuss the Heisenberg group which is the model case of our results.

First, recall that the underlying manifold of the Heisenberg group is $M=\Real^{2n+1}$ with coordinates $\{x_1,...,x_n,y_1,...,y_n,z\}$. The contact form $\eta$ and the Reeb field $V$ are given by $\eta=dz-\frac{1}{2}\sum_{i=1}^{2n}x_idy_i+\frac{1}{2}\sum_{i=1}^{2n}y_idx_i$ and $V=\partial_z$, respectively. Let $X_i=\partial_{x_i}-\frac{1}{2}y_i\partial_z$ and $Y_i=\partial_{y_i}+\frac{1}{2}x_i\partial_z$. The tensor $\mathcal J$ is defined by $\cJ(X_i)=Y_i$, $\cJ(Y_i)=-X_i$, and $\cJ(V)=0$. The Riemannian metric $\left<\cdot,\cdot\right>$ is defined by the condition that $\{\frac{1}{\e}V, X_1,...,X_n,Y_1,...,Y_n\}$ is orthonormal. In this case, we have $\overline\Rm\equiv 0$. Below, we let $b=-\frac{\e|\nabla_H f_0|}{2}$ and $c=\frac{\left<\nabla f_0,V\right>}{2}$.

\begin{thm}\label{conjH}
Assume that $\gamma:[0,1]\to M$ is a minimizing geodesic between its endpoints. Then $|\left<\dot\gamma(0),V(\gamma(0))\right>|\leq 2\pi$.
\end{thm}

\begin{proof}
Let $\cF(t)=\left(
                     \begin{array}{cc}
                       \cF_1(t) & \cF_2(t)\\
                       \cF_2(t)^T & \cF_3(t)
                     \end{array}
                   \right)$, where $\cF_1(t)$ is a $3\times 3$ block. A computation shows that
\[
\begin{split}
&\tr \cF_1(t)\\
&=-\frac{(b^2+c^2)(\cos(2ct)-1)+2t^2b^2c^2\cos(2ct)-2tc^3\sin(2ct)}{t(b^2+c^2)(\cos(2ct)-1)+t^2b^2c\sin(2ct)}.
\end{split}
\]
The method of proof is the same as that of $\tr F_1(t)$ which can be found in the proof of Theorem \ref{MCPH}.

Since
\[
\begin{split}
&(b^2+c^2)(\cos(2ct)-1)+tb^2c\sin(2ct)\\
&=-2\sin(ct)((b^2+c^2)\sin(ct)-tb^2c\cos(ct)),
\end{split}
\]
$\tr F_1(t)$ blows up for some $t<1$ if $c>\pi$. Therefore, the result follows from Proposition \ref{conjS}.
\end{proof}

\begin{thm}\label{MCPH}
For each Borel set $U$ in the Heisenberg group,
\[
\begin{split}
&\int_{\varphi_t(U)}d\vol=\int_U\frac{\sin^{2n-2}(c(1-t))}{\sin^{2n-2}(c)}d\vol\\
&+\int_U\frac{(1-t)\sin(c(1-t))[(1-t)b^2c\cos(c(t-t))-(b^2+c^2)]}{\sin(c)[b^2c\cos(c)-(b^2+c^2)]}d\vol\\
&\geq (1-t)^{2n+3}\int_Ud\vol.
\end{split}
\]
\end{thm}

\begin{proof}
Let $F(t)=\left(
                     \begin{array}{cc}
                       F_1(t) & F_2(t)\\
                       F_2(t)^T & F_3(t)
                     \end{array}
                   \right)$, $G(t)=F(1-t)^{-1}$, and $G(t)=\left(
                     \begin{array}{cc}
                       G_1(t) & G_2(t)\\
                       G_2(t)^T & G_3(t)
                     \end{array}
                   \right)$, where $F_1(t)$ and $G_1(t)$ are $3\times 3$ blocks.

It follows that $G(0)=0$ and
\[
\begin{split}
&\dot G_1(t)=-G_1(t)R_1(1-t)G_1(t)-I\\
&-G_2(t)R_3(1-t)G_2(t)^T-W_1G_1(t)-G_1(t)W_1^T\\
&\dot G_2(t)=-G_1(t)R_1(1-t)G_2(t)-G_2(t)R_3(1-t)G_3(t)-W_1G_2(t)\\
&\dot G_3(t)=-G_2(t)^TR_1(1-t)G_2(t)-G_3(t)R_3(1-t)G_3(t)-I.
\end{split}
\]
Therefore, $G_2\equiv 0$ and
\[
\begin{split}
&\dot G_1(t)=-G_1(t)R_1(1-t)G_1(t)-I-W_1G_1(t)-G_1(t)W_1^T\\
&\dot G_3(t)=-G_3(t)R_3(1-t)G_3(t)-I.
\end{split}
\]

It follows that $F_2\equiv 0$. A computation using the method in \cite{Le} shows that
\begin{equation}\label{F1}
\begin{split}
&F_1(1-t)=G_1(t)^{-1}\\
&=\frac{1}{K_1(t)}\left(
                     \begin{array}{ccc}
                       \frac{c^2}{t} & \frac{bK_2(t)}{t} & b^3K_2(t) \\
                       \frac{bK_2(t)}{t} & \frac{-b^2K_2(t)+c^3t\cot(tc)}{t} & cb^2K_2(t) \\
                       b^3K_2(t) & cb^2K_2(t) & tb^2c^2-c\cot(tc)K_1(t) \\
                     \end{array}
                   \right),
\end{split}
\end{equation}
and
\begin{equation}\label{F3}
F_3(1-t)=G_3(t)^{-1}=-c\cot(ct)I,
\end{equation}
where
\[
\begin{split}
&K_1(t)=tcb^2\cot(tc)-b^2-c^2=b^2K_2(t)-c^2,\\
&K_2(t)=tc\cot(tc)-1.
\end{split}
\]

The following two inequalities and Proposition \ref{measure} give the result:
\begin{equation}\label{ineq1}
\begin{split}
&\tr F_1(1-t)\\
&=-\frac{(b^2+c^2)(\cos(2ct)-1)+2t^2b^2c^2\cos(2ct)-2tc^3\sin(2ct)}{t(b^2+c^2)(\cos(2ct)-1)+t^2b^2c\sin(2ct)}\\
&=-\frac{d}{dt}\ln\left(t(b^2+c^2)(\cos(2ct)-1)+t^2b^2c\sin(2ct)\right)\geq -\frac{5}{t}
\end{split}
\end{equation}
and
\begin{equation}\label{ineq2}
\begin{split}
\tr F_3(1-t)&=-(2n-2)c\cot(ct)\\
&=-\frac{d}{dt}\ln\sin^{2n-2}(ct)\geq -\frac{2n-2}{t}.
\end{split}
\end{equation}

The inequality (\ref{ineq2}) follows from $x\cot(x)\leq 1$ for all $x$ in the open interval $(-\pi,\pi)$ and Theorem \ref{conjH}.

For (\ref{ineq1}), we first minimize over $b$ and then over $c$ (using again Theorem \ref{conjH}) to obtain
\[
\begin{split}
&\tr F_1(1-t)\\&\geq-\lim_{b\to 0}\frac{(b^2+c^2)(\cos(2ct)-1)+2t^2b^2c^2\cos(2ct)-2tc^3\sin(2ct)}{t(b^2+c^2)(\cos(2ct)-1)+t^2b^2c\sin(2ct)}\\
&=-\frac{(\cos(2ct)-1)+2t^2c^2\cos(2ct)}{t(\cos(2ct)-1)+t^2c\sin(2ct)}\\
&\geq-\lim_{c\to 0}\frac{(\cos(2ct)-1)+2t^2c^2\cos(2ct)}{t(\cos(2ct)-1)+t^2c\sin(2ct)}=-\frac{5}{t}.
\end{split}
\]
\end{proof}

\smallskip

\section{Proof of Theorem \ref{main}}\label{proofmain}

Let $\tilde F_1(t)$ be the matrix defined by (\ref{F1}) and let
\[
\tilde f_3(t)=-(2n-2)c\cot(ct).
\]
Then
\[
\begin{split}
&\dot{\tilde F}_1(t)=-\tilde R_1(t)-\tilde F_1(t)^2-\tilde F_1(t)W_1-W_1^T\tilde F_1(t)\\
&\dot{\tilde f}_3(t)=-(2n-2)c^2-\frac{1}{2n-2}\tilde f_3(t)^2,
\end{split}
\]
respectively. Moreover, we also have $\tilde F_1^{-1}\to 0$ and $\frac{1}{\tilde f_3}\to 0$ as $t\to 1$.

Since $\bar R_1(t)\geq 0$, it follows that
\[
\begin{split}
&\frac{d}{dt}\left(F_1(t-\e)-{\tilde F}_1(t)\right)=\tilde F_1(t)^2-F_1(t-\e)^2+(\tilde F_1(t)-F_1(t-\e))W_1\\
&+W_1^T(\tilde F_1(t)-F_1(t-\e))-\bar R_1(t)-F_2(t-\e)F_2(t-\e)^T\\
&\leq(\tilde F_1(t)-F_1(t-\e))(W_1+\tilde F_1(t))+\left(W_1^T+F_1(t-\e)\right)(\tilde F_1(t)-F_1(t-\e)).
\end{split}
\]
Note also that $F_1(t-\e)\geq\tilde F_1(t)$ for all $t$ close enough to $1$. It follows from this and \cite[Proposition 1]{Ro} that
$F_1(t-\e)\geq\tilde F_1(t)$ for all $t\geq \e$. By letting $\e\to 0$, we obtain $F_1(t)\geq \tilde F_1(t)$ for all $t$ in $[0,1]$. Therefore, by (\ref{ineq1}),
\[
\tr F_1(t)\geq\tr \tilde F_1(t)\geq -\frac{5}{1-t}.
\]

Similarly, by using $\tr\bar R_3(t)\geq 0$, we also have
\[
\begin{split}
\frac{d}{dt}\tr F_3(t)&=-\tr R_3(t)-|F_3(t)|^2-\tr(F_2(t)^TF_2(t))\\
&\leq -(2n-2)c^2-\frac{1}{2n-2}(\tr F_3(t))^2.
\end{split}
\]

An argument as above shows that
\[
\tr F_3(t)\geq\tilde f_3(t)\geq-\frac{2n-2}{1-t}.
\]

Finally, the result follows from Proposition \ref{measure} and the above estimates on $\tr F_1(t)$ and $\tr F_3(t)$.

\smallskip

\section{Appendix}

In this appendix, we give the proofs of Proposition \ref{ind} and \ref{Rm}. They are very mild modification of the corresponding ones in the Sasakian case (see \cite{Bl}). First, we prove the following result for more general weakly Sasakian manifolds.

\begin{prop}
Assume that the structure $(J,V,\eta,\left<\cdot,\cdot\right>)$ is weakly Sasakian. Let $X_1$ and $X_2$ be vector fields contained in the distribution $\ker\eta$. Then
\begin{enumerate}
\item $\LD_V\cJ=0$,
\item $\LD_Vg(X_1,\cdot)=0$,
\item $\left<\nabla_{X_1}V,X_2\right>=-\left<\nabla_{X_2}V,X_1\right>$,
\item $\nabla_VV=\frac{\LD_V|V|^2}{2|V|^2}V+\frac{1}{2}J^2\nabla |V|^2$,
\item $\left<X_1,\cJ X_2\right>=\nabla_{X_1}\eta(X_2)-\nabla_{X_2}\eta(X_1)$,
\item $\nabla_{X_1}J(X_2)=\frac{\left<X_2,\cJ \nabla_{X_1}V\right>}{|V|^2}V$,
\item $\nabla_X\cJ(V)=-\cJ\nabla_XV$,
\item $\nabla_V\cJ(X)=\nabla_{\cJ X}V-\cJ \nabla_{X}V$,
\item $\nabla_V\cJ(V)=\frac{1}{2}\cJ\nabla|V|^2$.
\end{enumerate}
\end{prop}

\begin{proof}
By (\ref{Nij}), we have
\[
\cJ^2[V,X]=\cJ[V,\cJ X].
\]
It follows that $\cJ(\LD_V\cJ)X=0$ and the horizontal part of $(\LD_V\cJ)X$ vanishes for any $X$. Since $\eta\circ \cJ=0$, we have, by Cartan's formula, the following for the vertical part
\[
\eta((\LD_V\cJ)X)=-\LD_V\eta(\cJ X)=0.
\]
The first assertion follows.

By the first assertion and (\ref{conmet}), we have
\[
\LD_Vg(X_1,\cJ X_2)=\LD_V(d\eta)(X_1,X_2)=0.
\]
The second assertion follows.

The third and the fourth assertions follow from Koszul's formula. Assertion five follows from (\ref{conmet}).

By (\ref{Nij}), we have
\[
\begin{split}
&d\eta(X_1,X_2)V\\
&=-\cJ^2[X_1,X_2]+\cJ[\cJ X_1,X_2]+\cJ[X_1,\cJ X_2]-[\cJ X_1,\cJ X_2]\\
&=-\cJ^2\nabla_{X_1}X_2+\cJ^2\nabla_{X_2}X_1+\cJ\nabla_{\cJ X_1}X_2\\
&-\cJ\nabla_{X_2}(\cJ X_1)+\cJ\nabla_{X_1}(\cJ X_2) -\cJ\nabla_{\cJ X_2}X_1-\nabla_{\cJ X_1}(JX_2)+\nabla_{\cJ X_2}(\cJ X_1).
\end{split}
\]

Since $X_1$, $X_2$, and $X_3$ are in $\ker\eta$, it follows that
\[
\begin{split}
&\left<\nabla_{X_1}X_2-\nabla_{X_2}X_1-\nabla_{\cJ X_1}(\cJ X_2)+\nabla_{\cJ X_2}(\cJ X_1),X_3\right>\\
&=\left<\nabla_{\cJ X_1}X_2-\nabla_{X_2}(\cJ X_1)+\nabla_{X_1}(\cJ X_2) -\nabla_{\cJ X_2}X_1,\cJ X_3\right>\\
\end{split}
\]
In other words,
\[
\begin{split}
&\left<-(\nabla_{\cJ X_1}\cJ)X_2+(\nabla_{\cJ X_2}\cJ)X_1,X_3\right>=\left<-(\nabla_{X_2}\cJ)X_1+(\nabla_{X_1}\cJ)X_2,\cJ X_3\right>
\end{split}
\]

It follows that
\begin{equation}\label{est1}
\begin{split}
&\left<-(\nabla_{\cJ X_1}\cJ)X_2+(\nabla_{\cJ X_2}\cJ)X_1,\cJ X_3\right>\\
&=\left<(\nabla_{X_2}\cJ)X_1-(\nabla_{X_1}\cJ)X_2,X_3\right>.
\end{split}
\end{equation}

On the other hand, we have, by taking exterior derivative of $d\eta$, the following
\begin{equation}\label{est2}
\begin{split}
&\left<X_3,(\nabla_{X_2}\cJ)X_1\right>+\left<X_1,(\nabla_{X_3}\cJ)X_2\right>-\left<X_3,(\nabla_{X_1}\cJ)X_2\right>\\
&=\left<X_3,(\nabla_{X_2}\cJ)X_1\right>+\left<X_1,(\nabla_{X_3}\cJ)X_2\right>+\left<X_2,(\nabla_{X_1}\cJ)X_3\right>=0.
\end{split}
\end{equation}
By combining this with (\ref{est1}), we obtain
\[
\begin{split}
\left<X_1,(\nabla_{X_3}\cJ)X_2\right>&=\left<X_3,(\nabla_{X_1}\cJ)X_2\right>-\left<X_3,(\nabla_{X_2}\cJ)X_1\right>\\
&=\left<(\nabla_{\cJ X_1}\cJ)X_2-(\nabla_{\cJ X_2}\cJ)X_1,\cJ X_3\right>.
\end{split}
\]

By (\ref{est2}), we also have
\[
\left<X_1,(\nabla_{X_3}\cJ)X_2\right>=\left<\cJ X_3,(\nabla_{\cJ X_2}\cJ)X_1-(\nabla_{\cJ X_1}\cJ)X_2\right>.
\]
It follows that $\left<X_1,(\nabla_{X_3}\cJ)X_2\right>=0$.

A calculation shows that $\nabla_{X_1}\cJ(X_2)=\frac{\left<X_2,\cJ\nabla_{X_1}V\right>}{|V|^2}V$ for all tangent vectors $X_1$ and $X_2$ in $\ker\eta$. The sixth assertion follows. A similar calculation gives the seventh assertion. Using the formula at the beginning of this proof, we obtain
\[
\nabla_V\cJ(X)=\nabla_{\cJ X}V-\cJ\nabla_{X}V
\]
which is the eighth assertion. The last assertion follows from $\cJ V=0$.
\end{proof}

\begin{proof}[Proof of Proposition \ref{ind}]
The first nine assertions follows from the above Proposition. By Koszul's formula, $\left<\nabla_{X_1}X_2,X_3\right>$ is independent of $\e$ if $X_i$ are contained in $\ker\eta$. It follows that
\[
\begin{split}
\nabla_{X_1}X_2&=(\nabla_{X_1}X_2)_\hor-\frac{1}{\e^2}\left<X_2,\nabla_{X_1}V\right>V\\
&=(\nabla_{X_1}X_2)_\hor+\frac{1}{2}\left<X_2,\cJ X_1\right>V.
\end{split}
\]
This also gives $\nabla_YV=-\frac{\epsilon^2}{2}\cJ Y$ for all $Y$. Finally,
\[
\begin{split}
\left<\nabla_V X_1,X_2\right>&=\left<[V,X],X_2\right>-\frac{1}{2}\left<[X,X_2],V\right>\\
&=\left<[V,X],X_2\right>-\frac{\e^2}{2}\eta([X_1,X_2])\\
&=\left<[V,X],X_2\right>+\frac{\e^2}{2}d\eta(X_1,X_2)\\
&=\left<[V,X],X_2\right>-\frac{\e^2}{2}\left<\cJ X_1,X_2\right>
\end{split}
\]
for all sections $X_1$ and $X_2$ of $\ker\eta$. Therefore, the last assertion follows.

\end{proof}

\begin{proof}[Proof of Proposition \ref{Rm}]
By Proposition \ref{ind}, we have
\[
\begin{split}
&\nabla_{Y_1}\nabla_{Y_2}V=-\frac{\e^2}{2}\nabla_{Y_1}(\cJ Y_2)=-\frac{\e^2}{2}(\nabla_{Y_1}\cJ)Y_2-\frac{\e^2}{2}\cJ\nabla_{Y_1}Y_2\\
&=-\frac{\e^2}{2}(\nabla_{(Y_1)_H}\cJ)\left((Y_2)_\hor+\frac{\left<Y_2,V\right>}{\e^2}V\right)-\frac{\e^2}{2}\cJ\nabla_{Y_1}Y_2\\
&=-\frac{\e^2}{2}(\nabla_{(Y_1)_H}\cJ)(Y_2)_\hor-\frac{\left<Y_2,V\right>}{2}(\nabla_{(Y_1)_H}\cJ)V -\frac{\e^2}{2}\cJ\nabla_{Y_1}Y_2\\
&=-\frac{\e^2}{4}\left<Y_1,Y_2\right>V+\frac{\e^2\left<Y_2,V\right>}{4}(Y_1)_\hor -\frac{\e^2}{2}\cJ\nabla_{Y_1}Y_2.
\end{split}
\]
Therefore,
\[
\begin{split}
\Rm(Y_1,Y_2)V&=\nabla_{Y_1}\nabla_{Y_2}V-\nabla_{Y_2}\nabla_{Y_1}V-\nabla_{[Y_1,Y_2]}V\\
&=\frac{\e^2\left<Y_2,V\right>}{4}(Y_1)_\hor-\frac{\e^2\left<Y_1,V\right>}{4}(Y_2)_\hor.
\end{split}
\]
This is the first assertion.

If $X_1$ and $X_2$ are in $\ker\eta$, then
\[
\bar\nabla_{X_2}X_1=(\nabla_{X_2}X_1)_\hor=\nabla_{X_2}X_1-\frac{\left<\cJ X_2,X_1\right>}{2}V.
\]

It follows that
\[
\begin{split}
&\bar\nabla_{X_3}\bar\nabla_{X_2}X_1=\nabla_{X_3}\bar\nabla_{X_2}X_1-\frac{\left<\bar\nabla_{X_2}X_1,\cJ X_3\right>}{2}V\\
&=\nabla_{X_3}\left(\nabla_{X_2}X_1-\frac{\left<\cJ X_2,X_1\right>}{2}V\right)-\frac{\left<\nabla_{X_2}X_1,\cJ X_3\right>}{2}V\\
&=\nabla_{X_3}\nabla_{X_2}X_1+\frac{\e^2\left<\cJ X_2,X_1\right>}{4}\cJ X_3 -\frac{\left<\cJ\nabla_{X_3}X_2,X_1\right>}{2}V\\
&-\frac{\left<\cJ X_2,\nabla_{X_3}X_1\right>}{2}V -\frac{\left<\nabla_{X_2}X_1,\cJ X_3\right>}{2}V
\end{split}
\]
and
\[
\bar\nabla_{[X_2,X_3]}X_1=\nabla_{[X_2,X_3]}X_1-\frac{1}{2}\left<\cJ[X_2,X_3],X_1\right>V+\frac{\left<V,[X_2,X_3]\right>}{2}\cJ X_1.
\]

Therefore,
\[
\begin{split}
\overline{\Rm}(X_2,X_3)X_1&=\Rm(X_2,X_3)X_1+\frac{\e^2\left<\cJ X_3,X_1\right>}{4}JX_2 \\
&-\frac{\e^2\left<\cJ X_2,X_1\right>}{4}\cJ X_3-\frac{\e^2\left<\cJ X_2,X_3\right>}{2}JX_1.
\end{split}
\]
This gives the second assertion. The proofs of the remaining claims follow in a similar manner and are omitted.
\end{proof}

\end{document}